\pdfoutput=1
\documentclass[11pt,letterpaper]{amsart}
\usepackage{pinlabel}
\usepackage{amsmath} 
\usepackage{cjhebrew}
\usepackage{graphicx} 
\usepackage{subcaption}
\usepackage{tikz-cd}
\setkeys{Gin}{keepaspectratio}
\usepackage{hyperref}
\usepackage{url}
\usepackage{bm}
\usepackage{mathabx}
\usepackage[left=1.25in,top=1in,right=1.25in,bottom=1in,head=.1in]{geometry}
\usepackage{xcolor}
\usepackage{tikz}
\usepackage{adjustbox}
\usepackage[normalem]{ulem} %
\usepackage{enumerate}
\usepackage{hyperref}
\usepackage{comment}

\usepackage{verbatim}
\newcommand{\items}{\begin{itemize}[leftmargin=25pt,rightmargin=15pt]
  \setlength\itemsep{2pt}}

\usepackage[
backend=biber,
style=alphabetic,
sorting=nyt, doi=false, isbn=false,  eprint=false,maxnames=15,minnames=15
]{biblatex}

\newcommand{\stopitems}{\end{itemize}}

\usepackage[textsize=scriptsize]{todonotes}

\usepackage{fancyhdr}
\newtheorem{theorem}{Theorem}[section] 
\newtheorem*{theorem*}{Theorem}
\newtheorem{lemma}[theorem]{Lemma}

\newtheorem*{conjecture*}{Conjecture}
\newtheorem*{question*}{Question}
\newtheorem*{lemma*}{Lemma}
\newtheorem{proposition}[theorem]{Proposition}
\newtheorem{corollary}[theorem]{Corollary}
\newtheorem*{corollary*}{Corollary}
\theoremstyle{definition}
\newtheorem{definition}[theorem]{Definition}

\newtheorem{remark}[theorem]{Remark}

\newtheorem{example}[theorem]{Example}

\newtheorem*{example*}{Example}
\newtheorem*{remark*}{Remark}
\newtheorem*{remarks*}{Remarks}
\newtheorem*{addenda*}{Addenda}
\newtheorem*{construction*}{Construction}

\newcommand{\Z}{\mathbb Z}

\newcommand{\id}{\textup{id}}

\renewcommand{\phi}{\varphi}

\newcommand{\interior}{\textup{int}}

\newcommand{\unred}[1]{ \ignorespaces}  

\title{exotic diffeomorphisms on closed aspherical 4-manifolds}

\author{Hokuto Konno}
\address{Graduate School of Mathematical Sciences, the University of Tokyo, 3-8-1 Komaba, Meguro, Tokyo 153-8914, Japan \\and\\
RIKEN iTHEMS, Wako, Saitama 351-0198, Japan}
\email{konno@ms.u-tokyo.ac.jp}

\author{Jianfeng Lin}
\address{Yau Mathematical Sciences Center, Tsinghua University, Beijing, 100871, China}
\email{linjian5477@mail.tsinghua.edu.cn}

\author{Anubhav Mukherjee}
\address{Department of Mathematics, University of Southern California, Los Angeles, 90026, USA}
\email{anubhav.mukherjee@usc.edu}

\author{Juan Muñoz-Echániz}
\address{Department of Mathematics, State University of New York, Stony Brook, 11794, USA}
\email{juan.munoz-echaniz@stonybrook.edu}

\begin{document}
\setlength{\headheight}{12.0pt}

\begin{abstract}
We study a class of diffeomorphisms of closed manifolds obtained from Davis' reflection trick construction, which are obtained by reflecting a diffeomorphism supported in the fundamental chamber of such a manifold. We use this construction to exhibit the first examples of closed aspherical $4$-manifolds with diffeomorphisms that are topologically isotopic to the identity map, but not smoothly isotopic to it.
\end{abstract}
\maketitle
\section{Introduction}\label{section:introduction}

A smooth closed $d$-dimensional manifold $M$ is \textit{aspherical} if its universal cover $\widetilde{M}$ is contractible; prominent examples are non-positively curved manifolds, solvmanifolds and aspherical bundles. In this article, we study a new construction of \textit{diffeomorphisms} on closed aspherical manifolds, based on the \textit{Davis reflection trick} \cite{Dav83} and inspired by the recent work by Davis--Hayden--Huang--Ruberman--Sunukjian \cite{DHHRS}. Our construction applies in all dimensions $d$, but is particularly interesting in dimension $d=4$ as it yields the first examples of \textit{exotic diffeomorphisms} on closed aspherical manifolds in dimension $d=4$. The latter application crucially relies on our previous work \cite[Theorem D]{KLMME1}, which exhibited smooth contractible Mazur $4$-manifolds which admit non-local diffeomorphisms.

\subsection{Main applications}

The homeomorphism type of an aspherical manifold $M$ is expected to be determined by its fundamental group $\pi_1 (M)$; this is the \textit{Borel Conjecture}. In contrast, examples are known of exotic pairs (i.e. homeomorphic but not diffeomorphic) of closed aspherical manifolds in dimensions $d>4$ \cite{Wall99,BT25,FJ89}, and more recently also in dimension $d=4$ \cite{DHHRS}. (Thus, the fundamental group $\pi_1 (M)$ fails to control the diffeomorphism type of $M$ when $d \geq 4$.)


In the spirit of the Borel Conjecture, one can ask how the smooth mapping class group $\pi_0 \mathrm{Diff}(M)$ of a closed aspherical manifold $M$ compares with the outer automorphism group $\mathrm{Out}(\pi_1 (M))$ of its fundamental group. They are known to be isomorphic in dimensions $d\leq 3$ (\cite[Theorem 8.1]{farb-margalit}, \cite[\S 2.3]{AFW}). For arbitrary dimensions, it has been conjectured that the topological pseudo-isotopy mapping class group of $M$ is isomorphic to $\mathrm{Out}(\pi_1 (M))$ \cite{HLLR}. In contrast, examples are known of exotic diffeomorphisms (i.e. topologically but not smoothly isotopic to the identity map) on closed aspherical manifolds in dimensions $d>4$: on standard tori \cite{hsiang-sharpe,hatcher} and exotic tori \cite{BKKT}. We construct the first instances of this phenomenon in dimension $d=4$:

\begin{theorem}\label{thm:main_intro} There exist (infinitely-many) smooth closed orientable aspherical $4$-manifolds $M$ with a self-diffeomorphism $f : M \to M $ such that:
\begin{enumerate}
\item $f$ is topologically isotopic to the identity map
\item $f$ is smoothly pseudo-isotopic to the identity map
\item $f^n$ is not smoothly isotopic to the identity map, for all $n \neq0$.
\end{enumerate}
\end{theorem}

\noindent That is, Theorem \ref{thm:main_intro} yields a $\mathbb{Z}$ subgroup in the kernel of the canonical map
\begin{align}
\pi_0 \mathrm{Diff}(M) \to \pi_0 \mathrm{Homeo}(M) \times \pi_0 \widetilde{\mathrm{Diff}}(M).\label{canonicalmap}
 \end{align}
(Here $\pi_0 \mathrm{Homeo}(M), \pi_0 \mathrm{Diff}(M),\pi_0 \widetilde{\mathrm{Diff}}(M)$ stand for the topological, smooth and smooth pseudo-isotopy mapping class groups, respectively.) The kernel elements we detect are obtained by a composition of \textit{Dehn twists} along \textit{Seifert-fibered} $3$-manifolds smoothly embedded in a closed aspherical $4$-manifold $M$ obtained from the Davis reflection construction \cite{Dav83}; see \S \ref{subsection:intro_davis} below for a brief explanation. For context, we also mention that Budney--Gabai \cite{BG} have obtained a related result when $d=4$: for a closed hyperbolic $4$-manifold $M$, they show that the map 
\[
\pi_0 \mathrm{Diff}(M) \to  \mathrm{Out}(\pi_1 (M)) \times \pi_0 \widetilde{\mathrm{Diff}}(M)
\]
has a large kernel, but the kernel elements found therein (barbell diffeomorphisms) are \textit{not} trivial in $\pi_0 \mathrm{Homeo}(M)$.\\

A remarkable feature of Theorem \ref{thm:main_intro} is that it exhibits a phenomenon in smooth closed aspherical manifolds of dimension $d=4$ that cannot occur in smooth closed aspherical manifolds of high dimension. Namely, in our examples of Theorem \ref{thm:main_intro} we can ensure that $\pi_1 (M)$ is center-less and satisfies the Farrel--Jones conjectures. In turn, when $\pi_1 (M)$ has these two properties but $M$ is a closed smooth aspherical manifold of dimension $d\geq 9$, then the canonical map (\ref{canonicalmap}) is \textit{injective} (Proposition \ref{proposition:highdim}). See \S \ref{subsection:highdim} for more details.\\



Another salient aspect of Theorem \ref{thm:main_intro} stems from the fact (cf. Theorem \ref{thm:aspherical-irreducible}) that a smooth closed aspherical $4$-manifold $M$ is \textit{irreducible}: if $M$ is diffeomorphic to a connected sum $M_1 \# M_2$ then either $M_1$ or $M_2$ is a homotopy $4$-sphere. The question of whether closed irreducible $4$-manifolds can admit exotic diffeomorphisms has attracted a great deal of attention, with the first examples appearing only recently \cite{BKirreducible,KLMME3}. On the other hand, whether the diffeomorphisms in those examples have \textit{infinite} order in the smooth mapping class group remains unclear. In turn, Theorem \ref{thm:main_intro} yields the first examples with this property:
\begin{corollary}
There exist closed oriented irreducible $4$-manifolds that admit exotic diffeomorphisms with infinite order.
\end{corollary}


A further feature of our construction is that exoticity can persist under passage to covering spaces. Although the closed manifolds considered above are not simply connected, we obtain an exotic diffeomorphism that admits an exotic lift to every connected cover, including the universal cover. Since the underlying manifold is aspherical, its universal cover is a contractible open $4$-manifold. Thus the failure of smooth isotopy to the identity persists even after passing to a contractible covering space. More precisely, we prove the following.

\begin{theorem}
    \label{thm:exotic-on-every-cover}
There exists a closed irreducible aspherical smooth $4$-manifold $M$ and a diffeomorphism
$f\in \operatorname{Diff}(M)$ such that $f$ is topologically isotopic, but not smoothly isotopic, to
$\id_M$, with the following property. For every connected covering map $p:\widehat M\to M$, the diffeomorphism $f$ lifts to an exotic diffeomorphism on $\widehat M$.
\end{theorem}

\subsection{Diffeomorphisms and the Davis reflection construction}\label{subsection:intro_davis}

We now describe the construction of diffeomorphisms that we study and our main result (Theorem \ref{thm:main_injective_intro}) which is used to prove Theorem \ref{thm:main_intro}).

We consider Davis' reflection trick \cite{Dav83} (see \S\ref{section:background} for a more detailed summary). This takes as input a smooth compact $d$-manifold $X$, which we assume here to be oriented, with non-empty boundary, equipped with a flag triangulation $\mathcal{T}$ of the boundary manifold $\partial X$. One then considers the cell decomposition $\mathcal{T}^\ast$ of $\partial X$ which is dual to $\mathcal{T}$, and reflects $X$ across the faces of $\mathcal{T}^\ast$ to produce a non-compact oriented smooth $d$-manifold denoted $D(X)$. The manifold $D(X)$ is naturally acted upon by a right-angled Coxeter group $W$ with fundamental domain $X$. Choose a finite-index torsion-free subgroup $W_0 \leq W$ which acts preserving the orientation of $D(X)$ (as it turns out, the commutator subgroup $W_0 = [W,W]$ satisfies these assumptions). The quotient $Q(X) = D(X)/W_0$ yields a smooth closed oriented $d$-manifold. Crucially, $Q(X)$ is aspherical provided $X$ is.

A diffeomorphism $\tau$ of $X$ supported in its interior can similarly be reflected and then taken to the quotient: this yields a diffeomorphism of $Q(X)$ denoted $Q(\tau)$. The construction just described yields a natural map 
\begin{align}
\pi_0 \mathrm{Diff}(X, \partial X) \xrightarrow{Q} \pi_0 \mathrm{Diff}(Q(X)), \label{Q_intro}
\end{align}
with corresponding variants for the mapping class groups $\pi_0 \mathrm{Homeo}(X, \partial X), \pi_0 \widetilde{\mathrm{Diff}}(X, \partial X)$ also. We establish the following injectivity result for the map (\ref{Q_intro}) (see Theorem \ref{theorem:main1} and Corollary \ref{cor:injective} for more precise statements):
\begin{theorem}\label{thm:main_injective_intro}
Suppose both of the following conditions on $X$ hold:
\begin{enumerate}
\item $X$ smoothly embeds in a ball $B^d$ ($d = \dim X$), and
\item either $X$ is contractible with non-simply-connected boundary, or $H_k (X, \mathbb{Z}) \neq 0$ in some degree $k >0$.
\end{enumerate}
Then
\begin{itemize}
\item if $d\neq 4$, the kernel of (\ref{Q_intro}) coincides with the subgroup of \textit{local} diffeomorphisms: diffeomorphisms of $X$ supported in a smoothly embedded ball $B^d \subset X$
\item if $d=4$, the kernel of (\ref{Q_intro}) is \textit{contained in} the subgroup of local diffeomorphisms.
\end{itemize}
\end{theorem}
In the case $d=4$ our previous work \cite[Theorem D]{KLMME1}, based on the family Seiberg--Witten invariants, gave examples of smooth contractible Mazur $4$-manifolds $X$ bounding non-trivial Brieskorn integral homology $3$-spheres for which the \textit{boundary Seifert-fibered Dehn twist} $\tau_{\partial X} \in \pi_0 \mathrm{Diff} (X, \partial X)$ is (1) topologically isotopic to the identity map relative to the boundary, (2) trivial as a pseudo-diffeomorphism of $X$ relative to the boundary, and (3) no non-zero power of $\tau_{\partial X}$ is smoothly isotopic relative to the boundary to a diffeomorphism supported in a ball $B^4 \subset X$. Here, items (1) and (2) are consequences of the results of Orson--Powell \cite[Theorem A and Theorem 3.1]{OP}. Furthermore, $X$ smoothly embeds in $B^4$ by the Mazur property. Applying Theorem \ref{thm:main_injective_intro} to this $X$ and $\tau = \tau_{\partial X}$ thus proves Theorem \ref{thm:main_intro} with $M = Q(X)$ and $f = Q(\tau_{\partial X} )$. (See \S \ref{subsection:examples} for other examples, and for the construction of infinitely-many examples of Theorem \ref{thm:main_intro}).

\subsection*{Acknowledgements} The authors thank Samuel Muñoz-Echániz for pointing out Lemma \ref{lemma:samu} and explaining its proof. The authors also thank Gheehyun Nahm and Rohil Prasad for holding valuable conversation regarding the article. HK is partially funded by JSPS KAKENHI Grant Numbers 25K00908, 26K22244. JL is partially supported by  National Key R\&D Program of China 2025YFA1017500 and NSFC 12271281. AM is partially funded by NSF DMS-2405270. 

\subsection*{AI Disclosure}
The literature search that went into the proof of Lemma \ref{lemma:center-FJ} was done using GPT 6-Astra.

\section{Background on the Davis reflection construction}\label{section:background}

In this section we review some aspects of the Davis reflection construction that we shall use, following
Davis~\cite{Dav83} and Davis--Hayden--Huang--Ruberman--Sunukjian
\cite{DHHRS}. Our notation follows closely that in \cite{DHHRS}.

\subsection{The Davis manifold}

Let $X$ be a compact smooth $d$-dimensional manifold with non-empty compact boundary $\partial X$. Let $\mathcal{T}$
be a smooth triangulation of $\partial X$, with $V$ denoting its set of
vertices of $T$. Following \cite{Dav83}, consider the right-angled Coxeter group
\[
W=
\left\langle
v\in V
\ \middle|\
v^2=1,\ (vw)^2=1
\text{ if $v,w$ are joined by an edge of $\mathcal{T}$}
\right\rangle .
\]

Let $\mathcal{T}^\prime$ denote the smooth cell decomposition of $\partial X$ which is \textit{dual} to the triangulation $\mathcal{T}$. This gives $X$ the structure of a smooth manifold-with-corners, where the top dimensional strata in $\partial X$ are the top dimensional cells of $\mathcal{T}^\prime$. For each vertex $v\in V$, let $X_v$ denote the closed cell of $\mathcal{T}^\prime$
dual to the vertex $v$. For $x\in\partial X$, let $W_x \leq W$ be the subgroup generated by all vertices $v \in V$ with $x \in X_v$. Following \cite[\S 13]{Dav83}, we consider the following

\begin{definition}
The \emph{Davis complex} associated to $X$ is the quotient space 
\[
D(X)=(W\times X)/{\sim} \, ,\]
where $(h,x)\sim(g,y)$ if $ x=y $ and $ h^{-1}g\in W_x.$
\end{definition}

In other words, $D(X)$ is obtained by
reflecting a copy of $X$ across each face of the cell decomposition $\mathcal{T}^\prime$. It is shown in \cite[\S 17]{Dav83} (see also~\cite[\S 3.2]{DHHRS}) that $D(X)$ is naturally a smooth non-compact manifold, with a smooth proper action of the Coxeter group $W$ acting with $X$ as a closed fundamental domain (the action of $w \in W$ sends $[(g,x)] \in D(X)$ to $[(wg,x)] \in D(X)$). Thus, we can write $D(X)$ as a union of `chambers'
\[
D(X)=\bigcup_{g\in W}gX.
\]

Let $L\colon W\rightarrow \mathbb{N}$ denote the word length with respect to the generating set $V$ of $W$. When $X$ is oriented then $D(X)$ inherits a canonical orientation from the orientation of its fundamental chamber $X$. We note that the chamber $gX \subset D(X)$ is canonically identified as a smooth manifold oriented by the orientation of $D(X)$, with either $X$ or $-X$ according as to whether $L (g)$ is even or odd. 

It will be convenient to order the chambers $gX$ as follows: choose a total ordering
\[
W=\{g_0 = 1, g_1,g_2,\ldots\}
\]
such that $i \leq j \implies L (g_i) \leq L (g_j )$, and set $X_i=g_i X$. We define $P_n :=\bigcup_{i=0}^nX_i$, which is a compact codimension zero submanifold-with-corners of $D(X)$ (see \cite[\S 3.2]{DHHRS}). Thus, we obtain an exhaustion of $D(X)$ by smooth compact manifolds-with-corners:
\[
D(X)=\bigcup_{n\geq 0}P_n \,.
\]

The next result, Proposition \ref{prop:boundarysum}, concerns a decomposition of $P_n$ that will be crucially used in our later arguments. Whenever we appeal to this decomposition, it will be understood that we regard $P_n$ as a smooth compact manifold-\textit{with-boundary} (which, by abuse notation, we continue to denote this as $P_n$) by  smoothing corners in the original $P_n$. The following result is obtained by a straightforward modification of the proof of \cite[Lemma 3.2]{DHHRS}, after smoothing corners.

\begin{proposition}\label{prop:boundarysum}
Let $\interior{X}$ be the complement of a  compact collar neighborhood of $\partial X \subset X$. Let $\interior{X}_{i}$ be the image of $\interior{X}$ under $g_{i}$. Suppose that $\mathcal{T}$ is flag\footnote{A smooth triangulation $\mathcal{T}$ is said to be \textit{flag} if every finite collection of pairwise adjacent vertices of $\mathcal{T}$ spans a simplex. This condition holds, for example, after taking the barycentric subdivision of any $\mathcal{T}$.}. Then for for every $n\geq 0$, there is a diffeomorphism of manifolds-with-boundary 
\[
\phi_n : P_n \longrightarrow X_0 \natural\cdots\natural X_n \,,
\]
where $\natural$ denotes boundary connected sum, such that the diagram 
\[
\begin{tikzcd}
 P_n \arrow{r}{\phi_n} & X_0 \natural \cdots \natural X_n\\
\interior{X_i} \arrow{u} \arrow{ru} & \, 
\end{tikzcd}
\]
commutes for each $i = 0 , \ldots n $, where $\interior{X_i} \to P_n$ and $\interior{X_i}\to X_1 \natural \cdots \natural X_n$ denote the canonical inclusions.
\end{proposition}

\subsection{A compact quotient}
To obtain a closed manifold, we now discuss a suitable covering space $D(X) \to Q(X)$. Recall that $W$ acts on $D(X)$ smoothly and properly but it does not act freely. The stabilizer of $x \in D(X)$ under this action is the subgroup $W_x \leq W$, which is a subgroup of $(\mathbb{Z}/2)^V$ and hence $W_x \cong (\mathbb{Z}/2)^k$ for some $k$. Choosing a finite-index torsion-free
subgroup $W_0\leq W$ (that such $W_0 \leq W$ exists is proved in
\cite[Corollary~5.2]{Dav83}; in fact, one can take $W_0=[W,W]$), we thus have that $W_0$ acts \textit{freely} on $D(X)$.

\begin{definition}
Let $Q(X)=D(X)/W_0$ be the smooth closed manifold obtained as the quotient of $D(X)$ by the smooth, proper and free action of $W_0 \leq W$.
 \end{definition}

We note that $Q(X)$ is closed since $W_0$ has
finite index in $W$. Furthermore, if $Q(X)$ is orientable if and only if $X$ is orientable and the subgroup $W_0 \leq W$ is \textit{orientation-preserving} in the sense that $W_0$ is contained in the kernel of the homomorphism $\chi : W \to \mathbb{Z}/2$ given by $\chi (g) = L(g) \, \textrm{mod} \, 2$. (For example, $W_0 = [W,W]$ is orientation-preserving.) In this case, $Q(X)$ inherits a canonical orientation from an orientation of the fundamental chamber $X_0 \equiv X$.

We have the following crucial result:
\begin{proposition}[Davis]
\label{prop:Davis-aspherical}
Suppose that $X$ is aspherical and $\mathcal{T}$ is flag. Then $Q(X)$ is aspherical.
\end{proposition}
\noindent Indeed, following \cite[\S 3.3]{DHHRS}, by Proposition~\ref{prop:boundarysum} there is a homotopy-equivalence
\[
P_n\simeq X_0\vee\cdots\vee X_n,
\]
and hence $P_n$ is aspherical. Since the $P_n$ give a compact exhaustion of $D(X)$, then it follows that $D(X)$ is aspherical. Finally, since there is a covering space
$D(X)\to Q(X)$ then $Q(X)$ is aspherical.

\begin{example}[The $n$-torus] Let $X=D^n$, which we identify with the unit cube $X=[0,1]^n\subset \mathbb{R}^n$. Let $\mathcal{T}$ be the triangulation of $\partial X$ isomorphic to the boundary of the $n$-dimensional cross-polytope, with vertices $a_i$ and $b_i$, for $i=1,\ldots,n$, where $a_i$ and $b_i$ are the barycenters of the faces $\{x_i=0\}$ and $\{x_i=1\}$, respectively. Thus, a collection of these vertices spans a simplex precisely when it does not contain both $a_i$ and $b_i$ for any $i$. The $2n$ faces of the cube are then the $(n-1)$-dimensional cells of the dual cell decomposition $\mathcal{T}'$, with the cells dual to $a_i$ and $b_i$ being, respectively, the faces $\{x_i=0\}$ and $\{x_i=1\}$.

The associated right-angled Coxeter group is
\[
W=\left\langle a_1,b_1,\ldots,a_n,b_n\ \middle|\ a_i^2=b_i^2=1,\,[a_i,a_j]=[a_i,b_j]=[b_i,a_j]=[b_i,b_j]=1\text{ for }i\neq j\right\rangle .
\]
The generators $a_i$ and $b_i$ act by reflection across the corresponding faces of the cube. Repeatedly reflecting $X$ across its faces produces the standard tiling of $\mathbb{R}^n$ by unit cubes, and hence $D(X)=\mathbb{R}^n$.

Moreover,
\[
W=\langle a_1,b_1\rangle\times\cdots\times\langle a_n,b_n\rangle
\]
is identified with the direct product of $n$ copies of the infinite dihedral group. Taking
\[
W_0=\langle a_1b_1,\ldots,a_nb_n\rangle\cong\mathbb{Z}^n,
\]
the element $a_ib_i$ acts on $D(X)=\mathbb{R}^n$ as a translation of distance two in the $i$th coordinate direction. Thus $W_0$ is a finite-index, torsion-free and orientation-preserving subgroup of $W$, and quotienting out $D(X)$ by $W_0$ yields the $n$-torus $T^n$:
\[
Q(X)=D(X)/W_0=\mathbb{R}^n/(2\mathbb{Z})^n\cong T^n.
\]
\end{example}




\section{The aspherical manifold and its diffeomorphism}

We now discuss a class of diffeomorphisms of the Davis manifold $D(X)$ and its compact quotient $Q(X)$, state the main result of this article (Theorem \ref{theorem:main1}), and explain how this yields examples of exotic diffeomorphisms of closed aspherical $4$-manifolds.

\subsection{Diffeomorphisms of Davis manifolds}

In what follows we fix a smooth compact $d$-dimensional manifold $X$ with $\partial X \neq \emptyset$, a triangulation $\mathcal{T}$ of $\partial X$, and a torsion-free finite-index subgroup $W_0$ of the Coxeter group associated with $\mathcal{T}$. Let $X_0$ denote the chamber of $D(X) = \bigcup_{g \in W} gX$ corresponding to the identity element $g = 1$. 

\begin{definition}
If $f \in \mathrm{Diff}(X,\partial X)$ is a diffeomorphism of $X$, equal to the identity near the boundary, we associate to it three diffeomorphisms
\[
D(f)_0 \in \mathrm{Diff}_c (D(X)), \quad D(f)\in \mathrm{Diff}(D(X)), \quad Q(f) \in \mathrm{Diff}(Q(X)),
\]
constructed as follows. Let $D(f)_0 \in \mathrm{Diff}_c (D(X))$ be the compactly-supported diffeomorphism of $D(X) = \bigcup_{g \in W} gX$ which preserves each chamber of $D(X)$, acts on $X_0$ as $f$, and acts as the identity on every other chamber $gX$ ($g \neq 1$). Define $D(f) \in \mathrm{Diff}(D(X))$ to be the diffeomorphism of $D(X)$ with non-compact support which acts as $f$ on each chamber $gX \equiv X$ of $D(X)=\bigcup_{g \in W} gX$. In terms of the $W$ action on $D(f)$, we have
\[
D(f) = \prod_{g \in W} g \circ D(f)_0 \circ g^{-1} .
\]
Since $D(f)$ commutes with the $W$-action on $D(X)$ then $D(f)$ descends to a unique diffeomorphism $Q(f) \in \mathrm{Diff}(Q(X))$ of the quotient manifold $Q(X) = D(X)/W_0$. 

\end{definition}

\subsection{Main result}
Our main result is the following:

\begin{theorem}\label{theorem:main1}
Let $X$ be a smooth compact $d$-dimensional manifold with non-empty boundary, with a choice of flag triangulation $\mathcal{T}$ of $\partial X$ and torsion-free finite-index subgroup $W_0$ of the Coxeter group $W$ associated to $\mathcal{T}$. Suppose both of the following conditions on $X$ hold:
\begin{enumerate}
\item There exists a smooth embedding of $X$ in the $d$-dimensional ball $B^d$.
\item Either $H_k (X, \mathbb{Z} ) \neq 0$ in some degree $k >0$, or $X$ is contractible and $\pi_1 (\partial X) \neq \{1\}$.
\end{enumerate}
Suppose $f \in \mathrm{Diff}(X, \partial X )$ is a diffeomorphism of $X$ which is not smoothly isotopic, relative to the boundary, to a diffeomorphism supported inside a smoothly embedded $d$-dimensional ball $B^d \subset X$.

Then $Q(f)$ is not smoothly isotopic to a diffeomorphism of $Q(X)$ of the form $Q(\ell)$, where $\ell \in \mathrm{Diff}(X, \partial X)$ is a diffeomorphism supported inside a smoothly embedded $d$-dimensional ball $B^d \subset X$. (In particular, $Q(f)$ is not smoothly isotopic to the identity diffeomorphism $\mathrm{Id}_{Q(X)} = Q(\mathrm{Id}_X)$.)
\end{theorem}

The proof of Theorem \ref{theorem:main1} is given in \S \ref{section:proofmain}. We now discuss some consequences of this result.\\

The statement of Theorem \ref{theorem:main1} can be equivalently phrased as follows. Fix an orientation of $X$ (by assumption (1)). An orientation-preserving embedding of a ball $ B^d \subset X$ induces a homomorphism
\[
e_\ast : \pi_0 \mathrm{Diff}(B^d , \partial B^d) \to \pi_0 \mathrm{Diff}(X , \partial X)
\]
which is easily seen to be independent of the choice of $e$. The subgroup of $\pi_0 \mathrm{Diff}(X, \partial X)$ of \textit{local diffeomorphisms} is, by definition, the central (and hence normal) subgroup given by the image of $e_\ast$, and we denote it 
\[
( \pi_0 \mathrm{Diff}(X, \partial X))^{local} \leq \pi_0 \mathrm{Diff}(X, \partial X).
\]
Now, the $Q$ construction gives a natural homomorphism of mapping class groups
\[
Q : \pi_0 \mathrm{Diff}(X, \partial X) \to \pi_0 \mathrm{Diff}(Q(X)).
\]
Note that the image of the subgroup $( \pi_0 \mathrm{Diff}(X , \partial X))^{local}$ under $Q$ is a central subgroup of $\pi_0 \mathrm{Diff}(Q(X))$. Theorem \ref{theorem:main1} asserts that, under assumption (1) and (2) on $X$, the induced group homomorphism 
\begin{align}
Q : \frac{\pi_0 \mathrm{Diff}(X, \partial X)}{( \pi_0 \mathrm{Diff}(X , \partial X))^{local}} \to \frac{\pi_0 \mathrm{Diff}(Q(X))}{Q\big( ( \pi_0 \mathrm{Diff}(X , \partial X))^{local}\big)       } \label{Qmapquotient}
\end{align}
is \textit{injective}.\\ 

The following result gives additional information about the subgroup $Q\big( ( \pi_0 \mathrm{Diff}(X , \partial X))^{local}\big)$ appearing in the target of the map (\ref{Qmapquotient}).
\begin{proposition}\label{prop:imageQlocal}
Suppose that $X$ is orientable and $W_0 \leq W$ is orientation-preserving. Let $r \in O(d)$ denote any linear reflection of $\mathbb{R}^d$. Then the index of $W_0$ in $W$ is even and
\[
Q\big( ( \pi_0 \mathrm{Diff}(X , \partial X))^{local}\big) = \Big\{ D_0 \Big( e_\ast \big( (\ell \cdot r\ell r^{-1})^{|W:W_0|/2} \big)\Big) \, | \, \ell \in  \pi_0 \mathrm{Diff}(B^d , \partial B^d) \Big\}.
\]
\end{proposition}
\begin{proof}
Let $\ell \in \mathrm{Diff}(B^d, \partial B^d)$, and fix an embedding of $B^d \subset X$ in the interior. Then the diffeomorphism $Q(e_\ast (\ell)) \in \mathrm{Diff}(Q(X))$ is supported in the disjoint union of $|W:W_0|$ balls
\[
\bigsqcup_{g \in W/W_0} g\cdot B^d \subset \bigcup_{g \in W/W_0}g\cdot X = Q(X) .
\]

Fix a word-length compatible ordering of $W$, and consider the corresponding ordered chambers $X_0 , X_1 , \ldots \subset D(X)$. Let $B_i \subset X_i$ denote the ball $g_i \cdot B^d \subset g_i \cdot X = X_i$. For each $i$, we choose an isotopy of the ball $B_i$ from $X_i \subset D(X)$ to the ball $B_0$ in $X_0 \subset D(X)$ obtained by a concatenation of \textit{simple} isotopies: a simple isotopy is one which takes the ball from a chamber $X_j$ $(1 \leq j \leq i$) to $P_{j-1} = \bigcup_{k =0}^{j-1}X_k$ by crossing $\partial X_j$ along a face of $\mathcal{T}^\ast$ contained in $\partial P_{j-1}$.

Now, if $B^d \subset X$ is regarded as parametrised by an orientation-preserving embedding, then the balls $B_i \subset X_i$ inherit a parametrisation. However, with this parametrisation the orientation of $B_i$ coincides with that of $X_i$ if and only if the word-length $L(g_i)$ of $g_i$ is even. This means that after the isotopy of $B_i$ from $X_i$ to $B_0 \subset X_0$ the resulting parametrised ball in $X_0 $ has the orientation of $X_0$ if and only if $L(g_i)$ is even. Since $W_0 \leq W$ is orientation-preserving then $W_0$ only contains elements with even word-length, so the function $\chi (g) = L(g) \, \textrm{mod} \, 2 \in \mathbb{Z}/2$ is well-defined on the set of cosets: $\chi : W/W_0 \to \mathbb{Z}/2$. 

Regarding $Q( e_\ast (\ell ) )$ as the product of $|W:W_0|$ commuting diffeomorphisms each supported in a chamber $X_i$, the choices of isotopies of balls lead to isotopies in $Q(X)$ of each of the $|W:W_0|$ diffeomorphism in that factorisation. If the isotoped ball $B_i  \subset X_0$ has the same orientation as $X_0$ then the resulting diffeomorphism of $Q(X)$ is isotopic to $D_0 ( e_\ast \ell )$, whereas if the orientation is reversed then it is isotopic to $D_0 ( e_\ast (r\ell r^{-1}) )$. Overall, we've shown that $Q(e_\ast (\ell) )$ is isotopic to the diffeomorphism of $Q(X)$ supported in the interior of $X_0$ and given by
\[
D_0 \Big( e_\ast \big( \prod_{g \in W/W_0} r^{\chi(g)}\ell r^{-\chi(g)} \big)\Big).
\]

Finally, note that the subsets $\chi^{-1} (0)$ and $\chi^{-1}(1)$ of $W/W_0$ are in bijection: if $v \in W$ denotes any Coxeter generator, then left multiplication by $v$ on $\chi^{-1}(0) $ yields a bijection $\chi^{-1}(0)\to \chi^{-1}(1)$. Hence the sets $\chi^{-1}(0)$, $\chi^{-1}(1)$ have cardinality $|W:W_0|/2$ and the claimed result follows from our previous discussion.
\end{proof}

We obtain the following injectivity result, which implies Theorem \ref{thm:main_injective_intro}:
\begin{corollary}\label{cor:injective}
Let $X$ be a smooth compact $d$-dimensional manifold with non-empty boundary, with a choice of flag triangulation $\mathcal{T}$ of $\partial X$ and torsion-free finite-index orientation-preserving subgroup $W_0$ of the Coxeter group $W$ associated to $\mathcal{T}$. Suppose both of the following conditions on $X$ hold:
\begin{enumerate}
\item There exists a smooth embedding of $X$ in the $d$-dimensional ball $B^d$.
\item Either $H_k (X, \mathbb{Z} ) \neq 0$ in some degree $k >0$, or $X$ is contractible and $\pi_1 (\partial X) \neq \{1\}$.
\end{enumerate}
Suppose that either of the following holds:
\begin{itemize}
\item $d\neq4$, or
\item $d=4$, and $W_0\leq W$ is not orientation-preserving (and hence $Q(X)$ is non-orientable).
\end{itemize}
Then the $Q$ construction defines an injective group homomorphism
\[
 \frac{\pi_0 \mathrm{Diff}(X, \partial X)}{( \pi_0 \mathrm{Diff}(X , \partial X))^{local}} \hookrightarrow \pi_0 \mathrm{Diff}(Q(X)) \quad , \quad [f] \mapsto Q(f).
\]
\end{corollary}
\begin{proof}
By Theorem \ref{theorem:main1} and its re-phrasing as the injectivity of (\ref{Qmapquotient}), it suffices to show that $Q\big( ( \pi_0 \mathrm{Diff}(X , \partial X))^{local}\big)$ is the trivial subgroup of $\pi_0 \mathrm{Diff}(Q(X))$. A straightforward variant of the argument in the proof of Proposition \ref{prop:imageQlocal} to the case when $W_0 \leq W$ is not orientation-preserving, and hence $Q(X)$ is non-orientable, shows that $Q\big( ( \pi_0 \mathrm{Diff}(X , \partial X))^{local}\big)$ is trivial. Thus, assume $W_0$ is orientation-preserving from now on. From Proposition \ref{prop:imageQlocal}, the triviality of $Q\big( ( \pi_0 \mathrm{Diff}(X , \partial X))^{local}\big)$ is implied by the statement that the conjugation automorphism
\[
c : \pi_0 \mathrm{Diff}(B^d , \partial B^d ) \to \pi_0 \mathrm{Diff}(B^d , \partial B^d ) \quad , \quad \ell \mapsto r \ell r^{-1}
\]
equals the inversion map: $c(\ell) = \ell^{-1}$. Below we explain why conjugation equals inversion in all dimensions $d \neq 4$.\\

If $d< 4$ then $\pi_0 \mathrm{Diff}(B^d , \partial B^d)$ is trivial \cite{smale,hatchersmale}.\\

If $d>4$ then there is a canonical group isomorphism (\cite[Corollary 2]{cerf-pisotopy})
\[
\pi_0 \mathrm{Diff}(B^d , \partial B^d) \cong \Theta_{d+1}
\]
where $\Theta_{d+1}$ is the group of diffeomorphism-classes of smooth oriented homotopy $(d+1)$-spheres (with group operation given by connected sum). This isomorphism sends $\ell \in \pi_0 \mathrm{Diff}(B^d , \partial B^d)$ to the homotopy sphere $S_\ell$ obtained by gluing the oriented standard ball $B^{d+1}$ with $- B^{d+1}$ along the orientation-preserving diffeomorphism of $S^{d} = \partial B^{d+1}$ given by extending $\ell$ as the identity outside of a ball $B^d$ in $S^d$: denoting the extension also by $\ell$,
\[
S_\ell = B^{d+1} \bigcup_{\ell} - B^{d+1}.
\]
Now, let $r$ also denote the image of $r \in O(d)$ under the inclusion $O(d) \leq  O(d+1)$. Then we obtain an orientation-\textit{reversing} diffeomorphism from $S_\ell $ to $S_{r\ell r^{-1} }$ by sending $x \in \pm B^{d+1} \subset S_\ell$ to $r(x) \in \pm B^{d+1} \subset S_{r\ell r^{-1}}$. The orientation reversal of a homotopy $(d+1)$-sphere represents its \textit{inverse} in the group $\Theta_{d+1}$, and hence conjugation in $\pi_0 \mathrm{Diff}(B^d , \partial B^d)$ equals the inversion map, as claimed.\end{proof}

\begin{remark}On the other hand, it is not known whether conjugation by a reflection equals inversion in $\pi_0 \mathrm{Diff}(B^4 , \partial B^4)$. Thus, the argument showing $Q\big( ( \pi_0 \mathrm{Diff}(X , \partial X))^{local}\big) = 0$ in the proof of Corollary \ref{cor:injective} doesn't apply when $d=4$ and $W_0$ is not orientation-preserving. 

\end{remark}

\subsection{Exotic diffeomorphisms of closed aspherical $4$-manifolds}\label{subsection:examples}
We now describe examples of $X$ and $f$ satisfying the hypothesis of Theorem \ref{theorem:main1} where, in addition: $X$ is aspherical (and hence $Q(X)$ is aspherical by Proposition \ref{prop:Davis-aspherical}) and $\dim X = d = 4$. Our examples are inspired by those in \cite{DHHRS}.\\

Let $Y$ be a closed oriented $3$-manifold in the following family of Brieskorn homology spheres, previously considered in \cite[Theorem D]{KLMME1},
\[
\begin{cases}
\Sigma(p,ps+1,ps+2) \quad (s\geq 1, \,\, \text{and odd} \, p \geq 3),\\
\Sigma(p,ps-1,ps-2) \quad (s \geq 2, \, \, \text{and odd} \, p \geq 5).
\end{cases}
\]
These homology spheres belong in the larger Casson--Harer family, a well-known family of Brieskorn spheres known to bound contractible Mazur $4$-manifolds \cite{CassonHarer81}. Let $M$ be a contractible Mazur $4$-manifold with $\partial M = Y$. Let $\mathring{\Sigma}_g $ denote a punctured surface of genus $g \geq 1$. We attach the `genus $g$ $2$-handle' $\mathring{\Sigma}_g \times D^2$ to $M$ along a ($0$-framed) unknot $U$ in the boundary of $M$, and denote the resulting smooth compact oriented $4$-manifold-with-boundary by

\begin{align}
X := M \bigcup_{\nu (U) \cong (\partial \mathring{\Sigma}_g ) \times D^2} \mathring{\Sigma}_g \times D^2  . \label{X:example}
\end{align}

For the diffeomorphism, we let $\tau_Y \in \pi_0 \mathrm{Diff}(X, \partial X)$ denote the (isotopy class of the) Dehn twist diffeomorphism along the smoothly embedded Seifert-fibered $3$-manifold $Y =\partial M \subset X$ (see \cite[\S 2.1.1]{KLMME1} for the definition of Seifert-fibered Dehn twist). Then we take 
\begin{align}
f := (\tau_Y)^d \in \pi_0 \mathrm{Diff}(X, \partial X) \quad \text{with} \,\, d \neq 0.\label{f:example}
\end{align}

\begin{proposition}\label{Prop:3.6}
With $X$ and $f$ as in (\ref{X:example}) and (\ref{f:example}) we have:
\begin{enumerate}
\item $X$ is homotopy equivalent to the closed genus $g (\geq 1)$ surface $\Sigma_g$. (In particular, $X$ is aspherical and $H_k (X, \mathbb{Z})\neq 0$ for $k =1,2$.)
\item $X$ embeds smoothly in $B^4$.
\item $f$ is both topologically isotopic and smoothly pseudo-isotopic to the identity map on $X$, relative to the boundary. However, $f$ is not smoothly isotopic, relative to the boundary, to a diffeomorphism of $X$ supported inside a smoothly embedded $4$-ball $B^4 \subset X$.
\end{enumerate}
\end{proposition}
\begin{proof}
(1) Since $M$ is contractible then $X$ is homotopy equivalent to $\mathring{\Sigma}_g \times D^2$ with the cone on $(\partial \mathring{\Sigma}_g) \times D^2$ attached. This shows that $X \simeq \Sigma_g$.\\

(2) Note that any surface $\Sigma_g$ can smoothly embed in a 4-ball $B^4$, and thus there is an obvious embedding of $\Sigma\times D^2$ in 4-ball. So by adding handles on $X$ we can smoothly embed $X$ in $M\natural B^4$. Now notice that $M$ is a Mazur type 4-manifold, so the double of $M$ is diffeomorphic to $S^4$. And thus $X$ can be smoothly embedded in $B^4 \subset \mathbb{R}^4$. 

(3) Let $\tau_M\in\operatorname{Diff}(M,\partial M)$ denote the boundary
Dehn twist, supported in the interior of $M$. Then $f$ is the
extension of $\tau_M^d$ by the identity over the attached genus-$g$
handle. Since $M$ is contractible, \cite[Theorem A]{OP} implies
that $\tau_M^d$ is topologically isotopic to the identity relative
to $\partial M$. Extending this isotopy by the identity gives the
required topological isotopy on $X$.

We claim that there is a smooth embedding
\[
    e\colon X\hookrightarrow M
\]
which is the identity on a slightly shrunken copy of $M$ containing
$\operatorname{supp}(\tau_M^d)$. Indeed, since the attaching circle
is a $0$-framed unknot contained in a $3$-ball in $\partial M$, we
may identify
\[
    X\cong M\mathbin{\natural}(\Sigma_g\times D^2),
\]
preserving this shrunken copy of $M$. Embedding
$\Sigma_g\times D^2$ into $B^4$ gives an embedding
\[
    X\hookrightarrow M\mathbin{\natural}B^4\cong M,
\]
where the last identification can also be chosen to fix this copy
of $M$.

Suppose that $f$ were smoothly isotopic relative to $\partial X$
to a diffeomorphism $\ell$ supported in a smoothly embedded ball
$B\subset\operatorname{int}(X)$. Transporting this isotopy through
$e$ and extending by the identity gives a smooth isotopy relative
to $\partial M$ from $\tau_M^d$ to a diffeomorphism supported in
$e(B)$. This contradicts \cite[Theorem D]{KLMME1}, since $d\neq 0$.

Note that $f$ is smoothly pseudo-isotopic to the identity
relative to $\partial X$. Indeed, \cite[Theorem 3.1]{OP} gives
a smooth pseudo-isotopy from $\tau_M^d$ to the identity relative
to $\partial M$, which extends by the identity over the attached
handle.
\end{proof}


We will now show the existence of exotic diffeomorphism on aspherical 4-manifold. 

\begin{proof}[Proof of Theorem~\ref{thm:main_intro}]
For each $g\geq 1$, let $X_g$ be the manifold constructed in (4),
and let $\tau_g\in\operatorname{Diff}(X_g,\partial X_g)$ be the
extension of the boundary Dehn twist on the Mazur manifold.
Choose a flag triangulation of $\partial X_g$, and let
$W_{0,g}=[W_g,W_g]$ be the orientation-preserving, torsion-free,
finite-index subgroup of the associated Coxeter group. Set
\[
    M_g:=Q(X_g),\qquad F_g:=Q(\tau_g).
\]
By Proposition~\ref{Prop:3.6}(1) and Proposition 2.4, $M_g$ is aspherical.
By construction, it is smooth, closed, and orientable.

The proof of Proposition~\ref{Prop:3.6} gives a topological isotopy and a
smooth pseudo-isotopy from $\tau_g$ to the identity relative to
$\partial X_g$. These can be chosen to be the identity on a
collar of the boundary. Reflecting them chamberwise and taking
the quotient by $W_{0,g}$ shows that $F_g$ is topologically
isotopic and smoothly pseudo-isotopic to
$\operatorname{id}_{M_g}$.

For every $n\neq 0$, Proposition~\ref{Prop:3.6}(3), applied with $d=n$,
shows that $\tau_g^n$ is not smoothly isotopic relative to the
boundary to a diffeomorphism supported in a ball. The remaining
hypotheses of Theorem~\ref{theorem:main1} hold by Proposition~\ref{Prop:3.6}(1)--(2).
Consequently,
\[
    F_g^n=Q(\tau_g^n)
\]
is not smoothly isotopic to the identity.

To obtain infinitely many homeomorphism types, consider the
folding map
\[
    r_g\colon M_g\longrightarrow X_g,
\]
which identifies each chamber with $X_g$. This is well-defined
because the reflection identifications preserve the
$X_g$-coordinate. If $i_g\colon X_g\hookrightarrow M_g$ denotes
the inclusion of the fundamental chamber, then
$r_g\circ i_g=\operatorname{id}_{X_g}$. Hence $(i_g)_*$ is
injective on homology, and
\[
    b_1(M_g)\geq b_1(X_g)=2g.
\]
Thus the first Betti numbers are unbounded, and a subsequence
of the $M_g$ consists of pairwise non-homeomorphic manifolds.

The contractible-chamber examples described in Section 1.2 are
obtained similarly by taking $X$ to be the original Mazur
manifold. Non-locality of every nonzero power of its boundary
Dehn twist follows from \cite[Theorem D]{KLMME1}, while
\cite[Theorem A and Theorem 3.1]{OP} give the required relative
topological isotopy and smooth pseudo-isotopy. The hypotheses
of Theorem~\ref{theorem:main1} hold because $X$ embeds in $B^4$ and
$\pi_1(\partial X)\neq 1$. These examples also satisfy the
conclusions of Lemma~\ref{lemma:center-FJ}.
\end{proof}

We will now prove the existence of an exotic diffeomorphism, whose every lift to any cover is still exotic. 
\begin{proof} [Proof of Theorem~\ref{thm:exotic-on-every-cover}]
Choose a contractible Mazur manifold $X$ with
$\partial X\not\cong S^3$, smoothly embedded in $B^4$, and a
diffeomorphism $\tau\in\operatorname{Diff}(X,\partial X)$
which is topologically isotopic to the identity relative to
the boundary, but is not smoothly isotopic relative to the
boundary to any diffeomorphism supported in a ball, as in
the construction of Theorem~\ref{thm:main_intro}.
Set $M=Q(X)$ and $f=Q(\tau)$. Since $X$ is contractible,
$D(X)$ is contractible and is the universal cover of $M$.
Moreover, $M$ is aspherical and hence irreducible by
Theorem~\ref{thm:aspherical-irreducible}.

Let $p\colon\widehat M\to M$ be any connected cover.
We may identify $\widehat M=D(X)/H$ for a subgroup
$H\leq W_0$. The chamberwise diffeomorphism $D(\tau)$
commutes with the $W$-action, so it descends to a lift
$\widehat f$ of $f$. Reflecting a boundary-relative
topological isotopy from $\operatorname{id}_X$ to $\tau$
and taking the quotient by $H$ shows that $\widehat f$
is topologically isotopic to the identity.

Suppose that $\widehat f$ were smoothly isotopic to the
identity. Lifting this isotopy to $D(X)$, starting at the
identity, gives a smooth isotopy whose endpoint is
$h\circ D(\tau)$ for some $h\in H$. By Lemma~\ref{lem: j and j' nonisotopic},
applied to a core of the fundamental chamber, we must
have $h=1$. Thus $D(\tau)$ is smoothly isotopic to the
identity. The cutoff and re-embedding argument in the
proof of Theorem~\ref{theorem:main1} then produces a smooth isotopy,
relative to $\partial X$, from $\tau$ to a diffeomorphism
supported in a ball: cut off the isotopy near the track
of the fundamental chamber, contain its support in a
finite boundary sum of chambers, and embed the other
summands into a ball in $X$. This contradicts the choice
of $\tau$. Hence $\widehat f$ is exotic.
\end{proof}

\subsection{Comparison with higher-dimensional case}\label{subsection:highdim}

We now draw attention to the following result about aspherical manifolds in high dimension which, as we shall explain below, is in sharp contrast with our results in dimension $4$:

\begin{proposition}\label{proposition:highdim}
 Let $M$ be a closed smooth aspherical manifold of dimension $d \geq 9$, whose fundamental group $\pi_1 (M)$ is center-less and satisfies the Farrel Jones conjectures (cf. \cite[\S 5.2]{HLLR}). Then the canonical map
\[
\pi_0 \mathrm{Diff}(M, \partial M) \to \pi_0 \mathrm{Homeo}(M, \partial M) \times \pi_0 \widetilde{\mathrm{Diff}}(M, \partial M)
\]
is injective.
\end{proposition}

The proof of Proposition \ref{proposition:highdim} will be given below. Before discussing that, we note that by Lemma \ref{lemma:center-FJ} below, it follows that in our examples of Theorem \ref{thm:main_intro} one can ensure that $\pi_1 (M)$ is center-less and satisfies the Farrel--Jones conjectures. Thus Theorem \ref{thm:main_intro} exhibits a phenomenon in aspherical manifolds of dimension $d=4$ which cannot occur in dimensions $d\geq 9$ by Proposition \ref{proposition:highdim}.

\begin{lemma}\label{lemma:center-FJ}
Let $X$ be a smooth compact simply-connected $d$--dimensional manifold
with connected nonempty boundary. Suppose that $d\geq 3$ and $\pi_1(\partial X)\neq\{1\}$. Let $\mathcal T$ be a finite flag PL triangulation of $\partial X$, let
$W=W_{\mathcal T}$ be the associated right-angled Coxeter group, and let
$W_0\leq W$ be a torsion-free subgroup of finite index. Then
\[
Z\bigl(\pi_1(Q(X))\bigr)=1,
\]
and $\pi_1(Q(X))$ satisfies the Farrell--Jones conjectures.
\end{lemma}

\begin{proof}
The simple connectivity of $X$
implies that $D(X)$ is simply connected. Hence $\pi_1(Q(X))\cong W_0$. 

We first prove that $W_0$ has trivial center. Recall that a Coxeter system
is irreducible if its Coxeter diagram is connected
\cite[Definition~3.5.3]{Davis-coxeter}. For a right-angled Coxeter group
with nerve $\mathcal T$, this is equivalent to saying that $\mathcal T$
is not a nontrivial simplicial join; see
\cite[\S3.5,4.1,7.1]{Davis-coxeter}. Suppose that $\mathcal T=\mathcal T_1*\mathcal T_2$ is a nontrivial simplicial join. If $\sigma_i$ is a maximal simplex of
$\mathcal T_i$, then $\operatorname{lk}_{\mathcal T}(\sigma_1)=\mathcal T_2$ and $\operatorname{lk}_{\mathcal T}(\sigma_2)=\mathcal T_1$. Since $\mathcal T$ is a PL triangulation of the closed manifold
$\partial X$, these links are PL spheres. Hence $\mathcal T_1$ and
$\mathcal T_2$ are PL spheres, and therefore so is
$\mathcal T=\mathcal T_1*\mathcal T_2$. Thus
$\partial X$ is a PL sphere, contradicting
$\pi_1(\partial X)\neq\{1\}$ since $d\geq3$. Therefore $W$ is
irreducible.

Since $\dim\mathcal T=d-1\geq2$, the classification of irreducible
spherical and affine Coxeter groups
\cite[Theorems~C.1.2 and~C.1.3]{Davis-coxeter} implies that the
right-angled Coxeter group $W$ is infinite and non-affine. By
\cite[Theorem~1.1]{Qi}, every finite-index subgroup of $W$ has trivial
center. Hence $Z\bigl(\pi_1(Q(X))\bigr)=Z(W_0)=1$.

Finally, by \cite[Theorem~12.3.3]{Davis-coxeter}, the Davis complex of $W$ is
CAT$(0)$. Since $\mathcal T$ is finite, $W$ acts properly and
cocompactly on this finite-dimensional CAT$(0)$ space, and so does the
finite-index subgroup $W_0$. Thus $W_0$ is a CAT$(0)$ group. The
$L$-theoretic Farrell--Jones conjecture for CAT$(0)$ groups follows
from \cite[Theorem~B]{BartelsLuck}, and the $K$-theoretic
Farrell--Jones conjecture follows from \cite{Wegner}. Therefore
$\pi_1(Q(X))\cong W_0$ satisfies both Farrell--Jones conjectures.
\end{proof}

\begin{proof}[Proof of Proposition \ref{proposition:highdim}]
If $M$ be is a closed aspherical $d$-dimensional of dimension $d\geq 4$ then, as explained in \cite[\S 5.2]{HLLR}, provided that $\pi_1 (M)$ is a group satisfying the Farrel--Jones conjectures then $M$ satisfies the ``identity block Borel conjecture'': 
\[
\pi_1 \widetilde{\mathrm{Homeo}}(M) \cong Z(\pi_1 (M)) \quad , \quad \pi_i  \widetilde{\mathrm{Homeo}}(M) = 0 \text{ for } i \geq 2 .
\]
Thus, assuming the center $Z(\pi_1 (M))$ is trivial then we have $\pi_1 \widetilde{\mathrm{Homeo}}(M) = \{ 1\}$. Thus, Proposition \ref{proposition:highdim} follows from Lemma \ref{lemma:samu} below. 
\end{proof}

\begin{lemma}[Communicated by Samuel Muñoz-Echániz]\label{lemma:samu}
Let $M$ be a smooth compact $d$-dimensional manifold (possibly $\partial M \neq \emptyset$). Suppose that $\pi_1 \widetilde{\mathrm{Homeo}}(M, \partial M) = \{1\}$ and $d\geq 9$. Then the canonical map
\[
\pi_0 \mathrm{Diff}(M, \partial M) \to \pi_0 \mathrm{Homeo}(M, \partial M) \times \pi_0 \widetilde{\mathrm{Diff}}(M, \partial M)
\]
is injective.
\end{lemma}
\begin{proof}
The proof uses the ``block-pseudoisotopy fiber sequence'', which we briefly recall. Let $F^{\mathrm{Diff}} (M)$, resp. $F^{\mathrm{Top}} (M)$, denote the homotopy fiber of the canonical map
\[
\mathrm{Diff}(M , \partial M) \to \widetilde{\mathrm{Diff}}(M , \partial M) \quad , \quad \text{resp.} \,\, \mathrm{Homeo}(M , \partial M) \to \widetilde{\mathrm{Homeo}}(M , \partial M).
\]
There are natural maps
\[
\begin{tikzcd}
    F^{\mathrm{Diff}}(M) \arrow{r}\arrow{d} &  \Omega^{\infty +2} \big( \mathrm{Wh}^{\mathrm{Diff}}(M)_{hC_2}\big) \arrow{d} \\
    F^{\mathrm{Top}}(M) \arrow{r} &  \Omega^{\infty +2} \big( \mathrm{Wh}^{\mathrm{Top}}(M)_{hC_2}\big)
\end{tikzcd}
\]
where $\mathrm{Wh}^{\mathrm{Diff}}(M)_{hC_2}$, resp. $\mathrm{Wh}^{\mathrm{Top}}(M)_{hC_2}$, denote the smooth (resp. topological) Whitehead spectra of $M$ after homotopy quotient by their natural $C_2$ action. 

As explained in \cite[Remark 5.4]{samu}\footnote{In \cite[Remark 5.4]{samu}, the word ``connective'' should be read ``connected''.} the map $\mathrm{Wh}^{\mathrm{Diff}}(M) \to \mathrm{Wh}^{\mathrm{Top}}(M)$ is $3$-connected\footnote{An ``$n$-connected map'' means that a continuous map which induces an isomorphism on $\pi_i$ for $i < n$ and a surjection on $\pi_n$.}. (This holds for any space $M$; in particular, the dimension $d$ is irrelevant here.) Hence, this map remains $3$-connected after passing to the homotopy quotient by the $C_2$ action. Hence the vertical arrow on the right in the above diagram is $1$-connected.

On the other hand, when $d\geq 5$ the work of Weiss--Williams \cite[Theorem A]{weiss-williams} shows that the horizontal arrows in the above diagram are $\phi(d)$-connected, where $\phi(d)$ is the concordance stable range (which is the same for the smooth and topological categories). By the work of Igusa \cite{igusa} we have
\[
\phi(d) \geq \min \big( \frac{d-4}{3} , \frac{d-7}{2}\big).
\]
Since $d\geq 9$ then $\phi(d) \geq 1$, and thus the horizontal arrows in the above diagram are $1$-connected.

Thus, it follows from the above discussion that the map $ F^{\mathrm{Diff}}(M)  \to  F^{\mathrm{Top}}(M) $ induces an isomorphism on $\pi_0$. We thus have the following commutative diagram where the rows are exact and there is an isomorphism in the vertical arrow indicated:
\[
\begin{tikzcd}
    \pi_1 \widetilde{\mathrm{Diff}}(M , \partial M) \arrow{r} \arrow{d}& \pi_0 F^{\mathrm{Diff}}(M) \arrow{r} \arrow{d}{\cong} & \pi_0 \mathrm{Diff}(M , \partial M) \arrow{r} \arrow{d} & \pi_0 \widetilde{\mathrm{Diff}}(M , \partial M )\arrow{d}\\ 
    \pi_1 \widetilde{\mathrm{Homeo}}(M , \partial M) \arrow{r} & \pi_0 F^{\mathrm{Top}}(M) \arrow{r}  & \pi_0 \mathrm{Homeo}(M , \partial M) \arrow{r}  & \pi_0 \widetilde{\mathrm{Homeo}}(M , \partial M ).
\end{tikzcd}
\]

Finally, using $\pi_1 \widetilde{\mathrm{Homeo}}(M , \partial M) = \{ 0 \}$ a quick diagram chase establishes from the last diagram the triviality of the kernel of the map
\[
\pi_0 \mathrm{Diff}(M, \partial M) \to \pi_0 \mathrm{Homeo}(M, \partial M) \times \pi_0 \widetilde{\mathrm{Diff}}(M, \partial M).
\]
\end{proof}

\section{Proof of Theorem \ref{theorem:main1}}\label{section:proofmain}
We now discuss a sequence of results leading up to a proof of Theorem \ref{theorem:main1}.\\

For the remainder of this section, $X$ is a smooth compact $d$-dimensional manifold $X$ with $\partial X \neq \emptyset$, $\mathcal{T}$ is a smooth triangulation of $\partial X$, and $W_0$ is a torsion-free finite-index subgroup of the Coxeter group $W$ associated with $\mathcal{T}$. 

We consider a pair of diffeomorphisms $f, \ell \in \mathrm{Diff}(X, \partial X)$ which equal the identity map on a neighborhood of the boundary, and such that $\ell$ is supported inside a $4$-ball $B^4 \subset X$. We begin discussing the following

\begin{proposition}\label{prop1}
Suppose that $\mathcal{T}$ is flag, and that either $H_k (X, \mathbb{Z}) \neq 0$ in some degree $k >0$, or $X$ is contractible and $\pi_1 (\partial X) \neq \{ 1\}$. Supose that $F_t \in \mathrm{Diff}(Q(X))$, t $\in [0,1]$, is a smooth isotopy with $F_0 = Q(\ell)$ and $F_1 = Q(f)$. Let $\widetilde{F}_t \in \mathrm{Diff}(D(X))$ be the unique smooth isotopy lifting $F_t$ with $\widetilde{F}_0 = D(\ell)$. Then $\widetilde{F}_1 = D(f)$.
\end{proposition}

Indeed, there is a unique lift $\widetilde{F}_t \in \mathrm{Diff}(D(X))$ of $F_t$ with $\widetilde{F}_0 = \mathrm{Id}_{D(X)}$ for the following reason. Let $\pi : D(X) \to Q(X)$ be the covering map. By the homotopy lifting properties of covering map, there exists a unique lift of the homotopy $F \circ \pi : D(X) \times [0,1] \to Q(X)$ to a homotopy $\widetilde{F} : D(X) \times [0,1]\to D(X)$ with $\widetilde{F}_0 = D(\ell)$. By the uniqueness of this lift, it follows that $\widetilde{F}_t$ is a diffeomorphism for all $t\in [0,1]$ (by considering the unique lifts of the constant homotopies $F_t \circ F_{t}^{-1} = F_{t}^{-1} \circ F_{t} = \mathrm{Id}_{Q(X)}$).

To prove Proposition \ref{prop1} we will need Lemma \ref{lem: j and j' nonisotopic}, discussed below. Fix a word-length compatible ordering of $W$, giving $g_0 = 1, g_1 , g_2 , \ldots$ and a corresponding ordering $X_0 =X , X_1 = g_1 X , X_2 = g_2 X  \ldots $ of the chambers of $D(X)$. Let $\interior{X}\subset X$ be the complement of a closed collar neighborhood of $\partial X\subset X$. Let $\mathring{X}$ be the closure of $\interior{X}$ in $X=X_0$. Without loss of generality, we may assume that $f(\mathring{X})=\mathring{X}$. We use $j: \mathring{X}\to D(X)$ to denote the canonical inclusion.  
\begin{lemma}\label{lem: j and j' nonisotopic} Let $X$ and $\mathcal{T}$ be as in Proposition \ref{prop1}. Given $g\in W$ with $g\neq 1$, use $j'$ to denote the composition 
\[
\mathring{X}\xrightarrow{f}\mathring{X}\xrightarrow{j}D(X)\xrightarrow{g}D(X).
\]
Then the embeddings $j$ and $j'$ are not smoothly isotopic. 
\end{lemma}
\begin{proof} Let $g=g_{a}$ for some $a>0$. Suppose for a contradiction that $j$ and $j'$ are smoothly isotopic. By compactness, there exists $N> a$ such that the embeddings $j: \mathring{X}\to P_{N}$ and $j'$ are isotopic through an ambient isotopy that fixes $\partial P_{N}$ pointwise. Let $P=X_{0}\natural  \cdots \natural X_N$. By Proposition \ref{prop:boundarysum}, there exists an ambient isotopy $\{\Psi_{t}:P\to P\}_{t\in [0,1]}$ on $P:=X_{0}\natural  \cdots \natural X_N$ that fixes $\partial P$ pointwise and satisfies $\Psi_{1}(\mathring{X}_0)=\mathring{X}_a$.

Suppose first that $H_{k}(X)\neq 0$ for some $k>1$. Then the map $H_{k}(\mathring{X}_0;\mathbb{Z})\to H_{k}(P;\mathbb{Z})$ and $H_{k}(\mathring{X}_a;\mathbb{Z})\to H_{k}(P;\mathbb{Z})$ have different images. So $\Psi_{t}$ cannot exist, and we arrive at a contradiction.

Finally, suppose $X$ is contractible and $\pi_1 (\partial X) \neq \{1\}$.  We consider the inclusion maps 
$\rho:\partial P\hookrightarrow P\setminus \mathring{X}_0$ and $\rho': \partial P\hookrightarrow P\setminus \mathring{X}_a$. Since $\rho'=\Psi_{1}\circ \rho$ and $\Psi_{t}|_{\partial P}=\operatorname{Id}$, the induced maps \[\rho_{*}: \pi_{1}(\partial P)\to \pi_{1}(P\setminus \mathring{X}_0) \text{ and }  \rho'_{*}: \pi_{1}(\partial P)\to \pi_{1}(P\setminus \mathring{X}_0)\] have the same kernel. Let $Y_i=\partial X_i$ and let $G_i=\pi_{1}(Y_i)$. Since $X$ is contractible,  we have  a homotopy equivalence \[P\setminus \mathring{X}_0\simeq (P\setminus \mathring{X}_0)\setminus (X_1\sharp \cdots \sharp X_N)\simeq X_0\setminus \mathring{X}_0\simeq Y_0\]
under which the inclusion map $\rho$ becomes the map 
\[
Y_0\sharp \cdots \sharp Y_{N}\to Y_0
\]
that collapses all the factors $Y_{i}$ with $i\neq 0$. Hence the map 
\[\rho_{*}:G_0\ast \cdots \ast G_N\cong \pi_{1}(\partial P) \to \pi_{1}(P\setminus \mathring{X}_0)\cong G_0\]
is the map that kills all the factors $G_{i}$ with $i\neq 0$. Similarly, the map $\rho'_{*}$  kills all the factors $G_{i}$ with $i\neq a$. So $\ker \rho_*\neq \ker \rho'_*$, contradicting again the existence of $\Psi_t$.
\end{proof}

\begin{proof}[Proof of Proposition \ref{prop1}] 
Since $D(f): D(X) \to D(X)$ is a lift of $Q(f): Q(X) \to Q(X)$, it follows that there exists a deck transformation $g \in W_0$ of $\pi : D(X) \to Q(X)$ such that $\widetilde{F}_1 = g \circ D(f)$.
So we need to show $g =1$.
Since $\widetilde F_1$ is isotopic to identity, its restriction to $\mathring{X}_0$ is isotopic to the embedding $j:\mathring{X}_0\to D(X)$. This restriction is exactly the composition 
\[
\mathring{X}\xrightarrow{f}\mathring{X}\xrightarrow{j}D(X)\xrightarrow{g}D(X).
\]
By Lemma \ref{lem: j and j' nonisotopic}, we must have $g=1$.
\end{proof}

\begin{proposition}\label{prop2}
Supose that $F_t \in \mathrm{Diff}(Q(X))$, t $\in [0,1]$, is a smooth isotopy with $F_0 = Q(\ell)$ and $F_1 = Q(f)$. Let $\widetilde{F}_t \in \mathrm{Diff}(D(X))$ be the unique smooth isotopy lifting $F_t$ with $\widetilde{F}_0 = D(\ell)$. Suppose $\widetilde{F}_1 = D(f)$. Then there exists a compactly-supported smooth isotopy $G_t \in \mathrm{Diff}_c (D(X))$ such that $G_0 = D(f)_0$ and $G_{1}|_{X_0} = \ell$.\end{proposition}
\begin{proof}
Again, we consider the compact subset
\[
K := \bigcup_{t \in [0,1]}\widetilde{F}_t (X_0 ) \subset D(X)
\]
and we cut off the isotopy $\widetilde{F}_t$ outside a neighborhood $\nu (K) \subset D(X)$ of $K$, as in the proof of Proposition \ref{prop1}, leading to an isotopy $H_t \in \mathrm{Diff} (D(X))$ such that
\begin{enumerate}
\item[(i)] $H_0 = D(\ell)$ and $H_{1}|_{X_0} = f$
\item[(ii)] $H_t = D(\ell)$ outside $\nu (K)$ for all $t \in [0,1]$.
\end{enumerate}
We then take
\[
G_t  := D(\ell) \circ H_{t}^{-1} \circ D(f)_0 \in \mathrm{Diff} (D(X)).
\]
Then $G_t$ is a compactly-supported isotopy of $D(X)$ (by (ii)) with $G_0 = D(f)_0$ and $G_{1}|_{X_0} = \ell$ (by (i)), as required.
\end{proof}

\begin{proof}[Proof of Theorem \ref{theorem:main1}]

Suppose for a contradiction that there exists a smooth isotopy $F_t \in \mathrm{Diff}(Q(X))$ from $F_0 = Q(\ell)$ to $F_1 = Q(f)$. Since $H_k (X,\mathbb{Z} ) \neq 0$ in some degree $k >0$, then by Proposition \ref{prop1} the unique lift $\widetilde{F}_t \in \mathrm{Diff}(D(X))$ with $\widetilde{F}_0 = D(\ell)$ has $\widetilde{F}_1 = D(f)$. Then by Proposition \ref{prop2} there exists a compactly-supported smooth isotopy $G_t \in \mathrm{Diff}_c (D(X))$ such that $G_0 = D(f)_0$ and $G_{1}|_{X_0} = \ell$. 

Fix a word-length compatible ordering of $W$, giving an ordering $X_0 , X_1 , \ldots$ of the chambers of $D(X)$. Let $n >0$ be an integer so large that $P_n =\bigcup_{i =0}^n X_i $ contains in its interior the support of the isotopy $G_t$, $t\in [0,1]$. By Proposition \ref{prop:boundarysum} we have a diffeomorphism of manifolds-with-boundary $\phi_n : P_n \xrightarrow{\cong} X_0 \natural X_1 \natural \cdots \natural X_n$ which intertwines the canonical inclusions of $\interior{X_i}$.

Since $X$ embeds smoothly in a ball $B^d$, embedding all the chambers $X_i$ with $i \neq 0$ into a ball $B^d$ yields a smooth embedding 
\[
\psi_n  : X_0 \natural X_1 \natural \cdots \natural X_n \hookrightarrow X_0 \natural B^d \natural \cdots \natural B^d = X_0 \equiv X.
\]
Furthermore, $\psi_n$ can be chosen to intertwine the canonical inclusions of $\interior{X_0}\equiv \interior{X}$. We thus obtain an embedding $\psi_n \circ \phi_n : P_n \hookrightarrow X$ which intertwines the inclusions of $\interior{X_0}\equiv\interior{X}$. Since the interior of $P_n$ contains the support of the isotopy $G_t$, this isotopy can be extended by the identity on $X \setminus \psi_n \circ \phi_n (P_n)$ to an isotopy $f_t \in \mathrm{Diff}(X, \partial X)$ of $X$ relative to its boundary. 

Since $G_0 = D(f)_0$ then $f_0 = f$. Since $G_1 \in \mathrm{Diff}_c (D(X))$ preserves $X_0$ and $G_{1}|_{X_0} = \ell$, then $f_1$ is supported inside the union of two disjoint smoothly embedded balls:
\begin{itemize}
    \item the ball $B^d \subset \interior{X_0} \hookrightarrow X_0 \natural B^d \natural \cdots \natural B^d =X_0  \equiv X$ containing the support of $\ell$
    \item the ball $ B^d \natural \cdots \natural B^d \hookrightarrow X_0 \natural B^d \natural \cdots \natural B^d = X_0 \equiv X$ complementary to the $X_0$ summand.
\end{itemize}
 Thus $f_t \in \mathrm{Diff}(X, \partial X)$ is an isotopy relative to the boundary from $f$ to a diffeomorphism supported in the union of two balls, and hence after a further isotopy it is supported inside a single ball. This contradicts assumption (2), and hence $Q(f)$ is not smoothly isotopic to a diffeomorphism supported inside a ball in $Q(X)$.
\end{proof}

\section{Aspherical $4$-manifolds are irreducible}

Recall that a closed connected smooth $4$-manifold $X$ is called {\it irreducible}
if, for every smooth connected-sum decomposition $X\cong M\#N$ into closed connected smooth $4$-manifolds, at least one of $M$ and $N$
is a homotopy $4$-sphere.
In this section, we shall prove the following:

\begin{theorem}\label{thm:aspherical-irreducible}
Every closed connected aspherical smooth $4$-manifold is irreducible.
\end{theorem}

First, we prove a few preliminary lemmas:

\begin{lemma}
\label{lem:free-product-cohomology}
Let $G$ and $H$ be groups. Then, for every $n\geq 1$,
\[
H^n(K(G*H,1);\Z/2) \cong H^n(K(G,1);\Z/2) \oplus H^n(K(H,1);\Z/2).
\]
\end{lemma}

\begin{proof}
Choose based CW models $A=K(G,1)$ and $B=K(H,1)$, with basepoints at
$0$-cells. By the Seifert--van Kampen theorem,
$\pi_1(A\vee B)\cong G*H$.
The universal cover of $A\vee B$ is a tree of spaces whose vertex spaces
are copies of the contractible universal covers $\widetilde A$ and $\widetilde B$
and whose edge spaces are points. It is therefore contractible, so
$A\vee B$ is a model for $K(G*H,1)$.
The conclusion follows from the cohomology decomposition for finite
wedges and the agreement of reduced and unreduced cohomology in positive
degrees.
\end{proof}

\begin{lemma}
\label{lem:free-product-factors-infinite}
Let $X$ be a closed connected aspherical smooth $4$-manifold, and suppose
that $\pi_1(X) \cong A*B$ for some groups $A, B$. Then each of $A$ and $B$ is either trivial or
infinite.
\end{lemma}

\begin{proof}
Since $X$ admits a finite-dimensional CW structure and is aspherical,
$\pi_1(X)$ is torsion-free (see, for example, \cite[Proposition~2.45]{Hatcher2002}).
The free factors $A$ and $B$ embed in $\pi_1(X)$ and are therefore
also torsion-free. A finite torsion-free group is trivial, so this completes the proof.
\end{proof}

\begin{lemma}\label{lem:homotopy-survives-connected-sum}
Let $X=M\#N$ be a connected sum of closed connected smooth $4$-manifolds.
There there are surjective homomorphisms 
$\pi_k(X) \to \pi_k(M)$, $\pi_k(X) \to \pi_k(N)$
for $k=2,3$.
\end{lemma}

\begin{proof}
Write $X = M^\circ \cup_{S^3} N^\circ$, where $M^\circ$ and $N^\circ$
are punctured 4-manifolds obtained from $M$ and $N$, respectively.
There is a continuous map $c_M : X\to M$ whose restriction to $M^\circ$
is the identity: map a collar of $\partial N^\circ$ radially onto the
removed ball in $M$ and collapse the rest of $N^\circ$ to its center.
Let $j_M : M^\circ \hookrightarrow X$ and $i_M : M^\circ \hookrightarrow M$
denote the inclusions. Then $c_M \circ j_M = i_M$.
Since $M$ is obtained from $M^\circ$ by attaching a single $4$-cell,
it follows that $(i_M)_* : \pi_k(M^\circ) \to \pi_k(M)$ is surjective for $k=2,3$ (see, for example, \cite[Corollary~4.12]{Hatcher2002}).
The factorization $(i_M)_* = (c_M)_*\circ(j_M)_*$ therefore shows that $(c_M)_* : \pi_k(X) \to \pi_k(M)$ is surjective in these degrees.
This proves the asserted surjectivity for $M$. Interchanging $M$ and $N$ gives the
remaining statements.
\end{proof}

\begin{proof}[Proof of Theorem~\ref{thm:aspherical-irreducible}]
Let $X$ be a closed connected aspherical smooth $4$-manifold, and suppose
that $X\cong M\#N$. By Lemma~\ref{lem:homotopy-survives-connected-sum}, we have
\begin{align}
\label{eq: vanishing pi2 pi3}
\pi_2(M)=\pi_3(M)=\pi_2(N)=\pi_3(N)=0.
\end{align}
Set $A=\pi_1(M)$ and $B=\pi_1(N)$.
The Seifert--van Kampen theorem gives $\pi_1(X)\cong A*B$.
By Lemma~\ref{lem:free-product-factors-infinite}, each of $A$ and $B$
is either trivial or infinite.

If $A$ is trivial, then $M$ is $3$-connected by \eqref{eq: vanishing pi2 pi3}. Since $M$ is a closed
connected orientable $4$-manifold, the Hurewicz theorem implies that
$M$ is an integral homology $4$-sphere and that
$\pi_4(M)\to H_4(M;\mathbb Z)$ is an isomorphism.
A map $S^4\to M$ representing a generator of $H_4(M;\mathbb Z)$ is
therefore a homology equivalence between simply-connected CW complexes.
By the homological form of Whitehead's theorem (for example, \cite[Corollary~4.33]{Hatcher2002}), this map $S^4\to M$ is a homotopy
equivalence.
Thus $M$ is a homotopy $4$-sphere. The same argument applies to $N$
when $B$ is trivial.

It remains to rule out the case in which both $A$ and $B$ are infinite.
In this case, the universal cover $\widetilde M$ is a simply-connected
noncompact smooth $4$-manifold with
$\pi_2(\widetilde M)=\pi_3(\widetilde M)=0$.
The Hurewicz theorem gives
$H_i(\widetilde M;\mathbb Z)=0$ for $1\leq i\leq 3$.
But $H_i(\widetilde M;\mathbb Z)=0$ for $i\geq 4$
by the non-comactness and dimensional reason, so $\widetilde M$ is acyclic.
Again, the homological form of Whitehead's theorem implies that $\widetilde M$ is contractible.
Thus $M$ is aspherical. The same argument shows that $N$ is aspherical.
Consequently, $X$, $M$, and $N$ are models for $K(A*B,1)$, $K(A,1)$,
and $K(B,1)$, respectively.
By Lemma~\ref{lem:free-product-cohomology}, we have $H^4(X;\Z/2) \cong H^4(M;\Z/2)\oplus H^4(N;\Z/2)$, which implies $\Z/2\cong\Z/2\oplus\Z/2$, a contradiction.

Hence at least one of $M$ and $N$ is a homotopy
$4$-sphere, as required. This completes the proof.
\end{proof}

\printbibliography

@book{Hatcher2002,
  AUTHOR     = {Hatcher, Allen},
  TITLE      = {Algebraic topology},
  PUBLISHER  = {Cambridge University Press, Cambridge},
  YEAR       = {2002},
  PAGES      = {xii+544},
  ISBN       = {0-521-79160-X; 0-521-79540-0},
  MRCLASS    = {55-01 (55-00)},
  MRNUMBER   = {1867354},
  MRREVIEWER = {Donald W. Kahn},
}

@article {OP,
    AUTHOR = {Orson, Patrick and Powell, Mark},
     TITLE = {Mapping class groups of simply connected 4-manifolds with
              boundary},
   JOURNAL = {J. Differential Geom.},
  FJOURNAL = {Journal of Differential Geometry},
    VOLUME = {131},
      YEAR = {2025},
    NUMBER = {1},
     PAGES = {199--275},
      ISSN = {0022-040X,1945-743X},
   MRCLASS = {57K40 (57N37 57R52 57S05)},
  MRNUMBER = {4947553},
MRREVIEWER = {Marja\ K.\ Kankaanrinta},
       DOI = {10.4310/jdg/1755544135},
       URL = {https://doi.org/10.4310/jdg/1755544135},
}

@misc{KLMME3,
      title={Constraints on Lefschetz fibrations with four-dimensional fibers from Seiberg-Witten theory}, 
      author={Hokuto Konno and Jianfeng Lin and Anubhav Mukherjee and Juan Muñoz-Echániz},
      year={2025},
      eprint={2511.00601},
      archivePrefix={arXiv},
      primaryClass={math.GT},
      url={https://arxiv.org/abs/2511.00601}, 
}

@article {BKirreducible,
    AUTHOR = {Baraglia, David and Konno, Hokuto},
     TITLE = {Irreducible 4-manifolds can admit exotic diffeomorphisms},
   JOURNAL = {Duke Math. J.},
  FJOURNAL = {Duke Mathematical Journal},
    VOLUME = {175},
      YEAR = {2026},
    NUMBER = {4},
     PAGES = {717--733},
      ISSN = {0012-7094,1547-7398},
   MRCLASS = {57K41 (57R50 57S05)},
  MRNUMBER = {5053586},
MRREVIEWER = {Marja\ K.\ Kankaanrinta},
       DOI = {10.1215/00127094-2025-0040},
       URL = {https://doi.org/10.1215/00127094-2025-0040},
}

@article {BG,
    AUTHOR = {Budney, Ryan and Gabai, David},
     TITLE = {On the automorphism groups of hyperbolic manifolds},
   JOURNAL = {Int. Math. Res. Not. IMRN},
  FJOURNAL = {International Mathematics Research Notices. IMRN},
      YEAR = {2025},
    NUMBER = {7},
     PAGES = {Paper No. rnaf083, 28},
      ISSN = {1073-7928,1687-0247},
   MRCLASS = {57S05 (57K40)},
  MRNUMBER = {4889279},
MRREVIEWER = {Rainer\ L\"owen},
       DOI = {10.1093/imrn/rnaf083},
       URL = {https://doi.org/10.1093/imrn/rnaf083},
}

@article {BKKT,
    AUTHOR = {Bustamante, Mauricio and Krannich, Manuel and Kupers,
              Alexander and Tshishiku, Bena},
     TITLE = {Mapping class groups of exotic tori and actions by {${\rm
              SL}_d({\bf Z})$}},
   JOURNAL = {Trans. Amer. Math. Soc. Ser. B},
  FJOURNAL = {Transactions of the American Mathematical Society. Series B},
    VOLUME = {11},
      YEAR = {2024},
     PAGES = {1316--1349},
      ISSN = {2330-0000},
   MRCLASS = {57R10 (11F06 20C12 57R52)},
  MRNUMBER = {4840293},
MRREVIEWER = {Thomas\ Koberda},
       DOI = {10.1090/btran/207},
       URL = {https://doi.org/10.1090/btran/207},
}

@incollection {hatcher,
    AUTHOR = {Hatcher, A. E.},
     TITLE = {Concordance spaces, higher simple-homotopy theory, and
              applications},
 BOOKTITLE = {Algebraic and geometric topology ({P}roc. {S}ympos. {P}ure
              {M}ath., {S}tanford {U}niv., {S}tanford, {C}alif., 1976),
              {P}art 1},
    SERIES = {Proc. Sympos. Pure Math.},
    VOLUME = {XXXII},
     PAGES = {3--21},
 PUBLISHER = {Amer. Math. Soc., Providence, RI},
      YEAR = {1978},
      ISBN = {0-8218-1432-X},
   MRCLASS = {57R52},
  MRNUMBER = {520490},
MRREVIEWER = {Gerald\ A.\ Anderson},
}

@article {hsiang-sharpe,
    AUTHOR = {Hsiang, W. C. and Sharpe, R. W.},
     TITLE = {Parametrized surgery and isotopy},
   JOURNAL = {Pacific J. Math.},
  FJOURNAL = {Pacific Journal of Mathematics},
    VOLUME = {67},
      YEAR = {1976},
    NUMBER = {2},
     PAGES = {401--459},
      ISSN = {0030-8730,1945-5844},
   MRCLASS = {57E05 (57D60 58D05)},
  MRNUMBER = {494165},
MRREVIEWER = {D.\ Burghelea},
       URL = {http://projecteuclid.org/euclid.pjm/1102817501},
}

@book {AFW,
    AUTHOR = {Aschenbrenner, Matthias and Friedl, Stefan and Wilton, Henry},
     TITLE = {3-manifold groups},
    SERIES = {EMS Series of Lectures in Mathematics},
 PUBLISHER = {European Mathematical Society (EMS), Z\"urich},
      YEAR = {2015},
     PAGES = {xiv+215},
      ISBN = {978-3-03719-154-5},
   MRCLASS = {57N10 (20F65 57-02 57M27)},
  MRNUMBER = {3444187},
MRREVIEWER = {Thomas\ Koberda},
       DOI = {10.4171/154},
       URL = {https://doi.org/10.4171/154},
}

@book {farb-margalit,
    AUTHOR = {Farb, Benson and Margalit, Dan},
     TITLE = {A primer on mapping class groups},
    SERIES = {Princeton Mathematical Series},
    VOLUME = {49},
 PUBLISHER = {Princeton University Press, Princeton, NJ},
      YEAR = {2012},
     PAGES = {xiv+472},
      ISBN = {978-0-691-14794-9},
   MRCLASS = {57M50 (20F36 20F65 57M07 57N05)},
  MRNUMBER = {2850125},
MRREVIEWER = {Stephen\ P.\ Humphries},
}

@article {HLLR,
    AUTHOR = {Hebestreit, Fabian and Land, Markus and L\"uck, Wolfgang and
              Randal-Williams, Oscar},
     TITLE = {A vanishing theorem for tautological classes of aspherical
              manifolds},
   JOURNAL = {Geom. Topol.},
  FJOURNAL = {Geometry \& Topology},
    VOLUME = {25},
      YEAR = {2021},
    NUMBER = {1},
     PAGES = {47--110},
      ISSN = {1465-3060,1364-0380},
   MRCLASS = {55R20 (55R40 55R60 57P10)},
  MRNUMBER = {4226228},
MRREVIEWER = {Semra\ Pamuk},
       DOI = {10.2140/gt.2021.25.47},
       URL = {https://doi.org/10.2140/gt.2021.25.47},
}

@article {FJ89,
    AUTHOR = {Farrell, F. T. and Jones, L. E.},
     TITLE = {Negatively curved manifolds with exotic smooth structures},
   JOURNAL = {J. Amer. Math. Soc.},
  FJOURNAL = {Journal of the American Mathematical Society},
    VOLUME = {2},
      YEAR = {1989},
    NUMBER = {4},
     PAGES = {899--908},
      ISSN = {0894-0347,1088-6834},
   MRCLASS = {53C20 (57R10 57R55 57R67)},
  MRNUMBER = {1002632},
MRREVIEWER = {Ronald\ J.\ Stern},
       DOI = {10.2307/1990898},
       URL = {https://doi.org/10.2307/1990898},
}

@article {BT25,
    AUTHOR = {Bustamante, Mauricio and Tshishiku, Bena},
     TITLE = {Symmetries of exotic aspherical space forms},
   JOURNAL = {Math. Z.},
  FJOURNAL = {Mathematische Zeitschrift},
    VOLUME = {309},
      YEAR = {2025},
    NUMBER = {3},
     PAGES = {Paper No. 47, 26},
      ISSN = {0025-5874,1432-1823},
   MRCLASS = {57S17},
  MRNUMBER = {4852262},
MRREVIEWER = {Rainer\ L\"owen},
       DOI = {10.1007/s00209-024-03673-8},
       URL = {https://doi.org/10.1007/s00209-024-03673-8},
}

@book {Wall99,
    AUTHOR = {Wall, C. T. C.},
     TITLE = {Surgery on compact manifolds},
    SERIES = {Mathematical Surveys and Monographs},
    VOLUME = {69},
   EDITION = {Second},
      NOTE = {Edited and with a foreword by A. A. Ranicki},
 PUBLISHER = {American Mathematical Society, Providence, RI},
      YEAR = {1999},
     PAGES = {xvi+302},
      ISBN = {0-8218-0942-3},
   MRCLASS = {57R67 (19J25 57-02)},
  MRNUMBER = {1687388},
       DOI = {10.1090/surv/069},
       URL = {https://doi.org/10.1090/surv/069},
}

@article {smale,
    AUTHOR = {Smale, Stephen},
     TITLE = {Diffeomorphisms of the {$2$}-sphere},
   JOURNAL = {Proc. Amer. Math. Soc.},
  FJOURNAL = {Proceedings of the American Mathematical Society},
    VOLUME = {10},
      YEAR = {1959},
     PAGES = {621--626},
      ISSN = {0002-9939,1088-6826},
   MRCLASS = {57.00},
  MRNUMBER = {112149},
MRREVIEWER = {G.\ T.\ Whyburn},
       DOI = {10.2307/2033664},
       URL = {https://doi.org/10.2307/2033664},
}

@article {hatchersmale,
    AUTHOR = {Hatcher, Allen E.},
     TITLE = {A proof of the {S}male conjecture, {${\rm Diff}(S^{3})\simeq
              {\rm O}(4)$}},
   JOURNAL = {Ann. of Math. (2)},
  FJOURNAL = {Annals of Mathematics. Second Series},
    VOLUME = {117},
      YEAR = {1983},
    NUMBER = {3},
     PAGES = {553--607},
      ISSN = {0003-486X},
   MRCLASS = {57M99 (57S05)},
  MRNUMBER = {701256},
MRREVIEWER = {R. C. Kirby},
       DOI = {10.2307/2007035},
       URL = {https://doi.org/10.2307/2007035},
}

@article {igusa,
    AUTHOR = {Igusa, Kiyoshi},
     TITLE = {The stability theorem for smooth pseudoisotopies},
   JOURNAL = {$K$-Theory},
  FJOURNAL = {$K$-Theory. An Interdisciplinary Journal for the Development,
              Application, and Influence of $K$-Theory in the Mathematical
              Sciences},
    VOLUME = {2},
      YEAR = {1988},
    NUMBER = {1-2},
     PAGES = {vi+355},
      ISSN = {0920-3036},
   MRCLASS = {57R50 (19D10 55P10 57R45 58C27)},
  MRNUMBER = {972368},
MRREVIEWER = {Jean-Claude\ Sikorav},
       DOI = {10.1007/BF00533643},
       URL = {https://doi.org/10.1007/BF00533643},
}

@article {weiss-williams,
    AUTHOR = {Weiss, Michael and Williams, Bruce},
     TITLE = {Automorphisms of manifolds and algebraic {$K$}-theory. {I}},
   JOURNAL = {$K$-Theory},
  FJOURNAL = {$K$-Theory. An Interdisciplinary Journal for the Development,
              Application, and Influence of $K$-Theory in the Mathematical
              Sciences},
    VOLUME = {1},
      YEAR = {1988},
    NUMBER = {6},
     PAGES = {575--626},
      ISSN = {0920-3036},
   MRCLASS = {57N35 (57R50 57R65)},
  MRNUMBER = {953917},
MRREVIEWER = {Christopher\ W.\ Stark},
       DOI = {10.1007/BF00533787},
       URL = {https://doi.org/10.1007/BF00533787},
}

@book{Davis-coxeter,
  author    = {Davis, Michael W.},
  title     = {The Geometry and Topology of Coxeter Groups},
  series    = {London Mathematical Society Monographs Series},
  volume    = {32},
  publisher = {Princeton University Press},
  address   = {Princeton, NJ},
  year      = {2008}
}

@article{Qi,
  author  = {Qi, Dongwen},
  title   = {On irreducible, infinite, non-affine Coxeter groups},
  journal = {Fundamenta Mathematicae},
  volume  = {193},
  number  = {1},
  pages   = {79--93},
  year    = {2007}
}

@article{BartelsLuck,
  author  = {Bartels, Arthur and L{\"u}ck, Wolfgang},
  title   = {The Borel conjecture for hyperbolic and {CAT}(0)-groups},
  journal = {Annals of Mathematics},
  series  = {2},
  volume  = {175},
  number  = {2},
  pages   = {631--689},
  year    = {2012}
}

@article{Wegner,
  author  = {Wegner, Christian},
  title   = {The {$K$}-theoretic Farrell--Jones conjecture for {CAT}(0)-groups},
  journal = {Proceedings of the American Mathematical Society},
  volume  = {140},
  number  = {3},
  pages   = {779--793},
  year    = {2012}
}

@incollection {cerf-pisotopy,
    AUTHOR = {Cerf, Jean},
     TITLE = {The pseudo-isotopy theorem for simply connected differentiable
              manifolds},
 BOOKTITLE = {Manifolds--{A}msterdam 1970 ({P}roc. {N}uffic {S}ummer
              {S}chool)},
    SERIES = {Lecture Notes in Math.},
    VOLUME = {Vol. 197},
     PAGES = {76--82},
 PUBLISHER = {Springer, Berlin-New York},
      YEAR = {1971},
   MRCLASS = {57.47},
  MRNUMBER = {290404},
MRREVIEWER = {S.\ Y.\ Husseini},
}

@article {samu,
    AUTHOR = {Mu\~noz-Ech\'aniz, Samuel},
     TITLE = {Mapping class groups of {$h$}-cobordant manifolds},
   JOURNAL = {Forum Math. Sigma},
  FJOURNAL = {Forum of Mathematics. Sigma},
    VOLUME = {13},
      YEAR = {2025},
     PAGES = {Paper No. e132, 38},
      ISSN = {2050-5094},
   MRCLASS = {57R80 (55R60 57N37 57Q10 57S05)},
  MRNUMBER = {4942994},
MRREVIEWER = {Semra\ Pamuk},
       DOI = {10.1017/fms.2025.10087},
       URL = {https://doi.org/10.1017/fms.2025.10087},
}

@article{Dav83,
  author  = {Davis, Michael W.},
  title   = {Groups generated by reflections and aspherical manifolds not covered by Euclidean space},
  journal = {Annals of Mathematics. Second Series},
  volume  = {117},
  number  = {2},
  year    = {1983},
  pages   = {293--324},
  doi     = {10.2307/2007079},
}

@article {CassonHarer81,
    AUTHOR = {Casson, Andrew J. and Harer, John L.},
     TITLE = {Some homology lens spaces which bound rational homology balls},
   JOURNAL = {Pacific J. Math.},
  FJOURNAL = {Pacific Journal of Mathematics},
    VOLUME = {96},
      YEAR = {1981},
    NUMBER = {1},
     PAGES = {23--36},
      ISSN = {0030-8730,1945-5844},
   MRCLASS = {57N10 (14J17)},
  MRNUMBER = {634760},
MRREVIEWER = {W.\ D.\ Neumann},
       URL = {http://projecteuclid.org/euclid.pjm/1102734944},
}

@article{DHHRS,
  author  = {Davis, Michael and Hayden, Kyle and Huang, Jingyin and Ruberman, Daniel and Sunukjian, Nathan},
  title   = {Exotic aspherical 4-manifolds},
  journal = {Annals of Mathematics},
  note    = {To appear},
  eprint  = {2411.19400},
  archivePrefix = {arXiv},
  primaryClass  = {math.GT},
  year    = {2024},
}

@article{KLMME1,
      title={On four-dimensional Dehn twists and Milnor fibrations}, 
      author={Hokuto Konno and Jianfeng Lin and Anubhav Mukherjee and Juan Muñoz-Echániz},
      year={2024},
      journal={arXiv:2409.11961},
}

\end{document}